\documentclass[11pt]{article}
\usepackage{amsmath,amsthm,amsfonts,amssymb,amscd, amsxtra,mathrsfs,commath,url}
\usepackage{cite,url}
\usepackage{graphicx}
\usepackage[title]{appendix}

\usepackage{multirow}
\usepackage{multicol}   
\usepackage[margin=2.7 cm,nohead]{geometry}
\usepackage{bbm}
\usepackage{xcolor}
\usepackage{float}
\usepackage{threeparttablex,color,soul,colortbl,hhline,rotating,booktabs,longtable,multirow,makecell}
\definecolor{lightgray}{gray}{0.95}
\usepackage{amssymb,amsfonts,amstext,amsmath}
\usepackage{cite}
\usepackage{graphicx,geometry}
\usepackage{listings}
\usepackage{algorithmic}
\usepackage{hyperref}

\newtheorem{theorem}{Theorem}
\newtheorem{lemma}[theorem]{Lemma}

\newtheorem{remark}{Remark}

\newcommand{\argmin}{{\rm arg}\!\min}

\newcommand{\J}{\mathcal{J}}
\newcommand{\N}{\mathbb{N}}
\newcommand{\R}{\mathbb{R}}

\newcommand{\inner}[2]{\langle{#1},{#2}\rangle}

\newcommand{\X}{\mathcal{X}}

\newcommand{\D}{\mathcal{D}}

\newcommand{\ds}{\displaystyle}

\newcommand{\vgap}{\vspace{.1in}}

\newcommand{\Rn}{\mathbb{R}^n}

\newcommand{\bi}{\begin{itemize}}
\newcommand{\ei}{\end{itemize}}
\newcommand{\ba}{\begin{array}}
\newcommand{\ea}{\end{array}}

\newcommand{\mgap}{\vspace{.1in}}

\usepackage{amssymb}

\begin{document}

\title{Improved convergence rates for the multiobjective Frank-Wolfe method}

\author{
    Douglas S. Gon\c{c}alves
    \thanks{Departamento de Matem\'atica, Universidade Federal de Santa Catarina, Florian\'opolis, SC 88040-900, Brazil. (E-mail: {\tt
       douglas.goncalves@ufsc.br}). The work of this author was supported in part by CNPq (Conselho Nacional de Desenvolvimento Cient\'{i}fico e Tecnol\'ogico)  Grant 305213/2021-0.}
  \and  Max L. N. Gon\c calves
  \thanks{IME, Universidade Federal de Goi\'as, Goi\^ania, GO 74001-970, Brazil. (E-mails: {\tt
       maxlng@ufg.br} and {\tt jefferson@ufg.br}). The work of these authors was
    supported in part by CNPq (Conselho Nacional de Desenvolvimento Cient\'{i}fico e Tecnol\'ogico) Grants  405349/2021-1,  304133/2021-3 and  312223/2022-6.}
   \and Jefferson G.  Melo\footnotemark[2]
}

\maketitle

\vspace{.5cm}

\begin{abstract} 
This paper analyzes the convergence rates of the {\it  Frank-Wolfe } method for solving convex constrained multiobjective optimization. We establish improved convergence rates under different assumptions on the objective function, the feasible set, and the localization of the limit point of the sequence generated by the method. In terms of the objective function values,  we firstly show that if the objective function is strongly convex and the limit point of the sequence generated by the method  lies in the relative interior of the feasible set, then the algorithm achieves a linear convergence rate. 
Next, we focus on a special class of problems where the feasible constraint set is $(\alpha,q)$-uniformly convex for some $\alpha >0$ and $q \geq 2$, including, in particular, \(\ell_p\)-balls for all $p>1$. In this context, we prove that the method attains:  (i) a  rate of $\mathcal{O}(1/k^\frac{q}{q-1})$ when the objective function is strongly convex; and  (ii) a linear rate (if $q=2$) or a rate of $\mathcal{O}(1/k^{\frac{q}{q-2}})$ (if $q>2$) under an additional assumption, which always holds if the feasible set does not contain an unconstrained weak Pareto point. We also discuss enhanced convergence rates for the algorithm in terms of an optimality measure. Finally, we provide some simple examples to illustrate the convergence rates and the set of assumptions.
\\[2mm]
{\bf keywords:} { Frank-Wolfe method; conditional gradient method;  convergence rate;  Pareto optimality; constrained multiobjective problem. }
\\[2mm]
\end{abstract}

\section{Introduction}\label{pre}
In this paper, we are interested in the constrained convex multiobjective optimization problem
\begin{equation} \label{vectorproblem}
\min_{x \in \mathcal{X}} F(x) = (F_1(x), \ldots, F_m(x)),
\end{equation}
where the constraint set $\mathcal{X}$ is a nonempty, convex, and compact subset of a finite dimensional vector space $\mathcal{Y}$ equipped with an inner product $\langle \cdot,\cdot \rangle$  and a norm $\|\cdot\|$. 

Many interesting problems in  applications  of engineering, economics, management science, and medicine can be formulated as \eqref{vectorproblem}; 
see, for example, \cite{Schreibmann,Eschenauer1990}. This motivates many researches in the last decades to develop algorithms for solving general multiobjective and vector optimization problems, many of them by extending algorithms from scalar ($m=1$) optimization. Specifically, \cite{mauricio&iusem} proposed and analyzed a projected gradient method for solving constrained vector optimization problems which are more general than \eqref{vectorproblem}. The convergence of this method was further explored in \cite{luis-jef-yun,fukuda2011convergence}. This method extends its unconstrained version proposed in  \cite{benar&fliege}. By adding the indicator function of the constraint set $\X$ to the objective function $F$, one can also apply multiobjective/vector proximal point methods such as the one introduced in \cite{bonnel2005proximal} and further developed in \cite{Glaydston1,iusem2021,fukudaProxGradMultiobjective,Ray1,acceProxGradMultiobjective,bento18b}. 
Newton-like methods \cite{bento18c} and other methods on manifolds \cite{bento12, bento13, bento13b, bento18} were also  proposed for multicriteria and vector optimization.

Recently, \cite{Leandro2022} introduced a Frank-Wolfe (conditional gradient) method for solving constrained multiobjective optimization problems. 
The study delved into various convergence properties of the method and showcased its numerical performance. 
Furthermore, \cite{pedro2023,ellen2024} explored extensions of the multiobjective Frank–Wolfe (M-FW) method  for solving composite multiobjective optimization problems, while \cite{chen2023} addressed its application to vector optimization problems. 
A multiobjective version of the away-step Frank-Wolfe method was also recently proposed and analyzed  in \cite{M-AFW}  for solving an instance of \eqref{vectorproblem}, where the constraint set $\X$ is a polyhedron. 

In this study our focus is on improved convergence rates that can be achieved by M-FW for solving convex constrained multiobjective optimization problems with special structure. We assume that the multiobjective function $F$ satisfies the following conditions:\\
{\bf(A1)} For every $j\in \J:=\{1,2,\ldots,m\}$, the component  $F_j\colon V\subset \mathcal{Y} \rightarrow \mathbb{R}^m$ is convex, continuously differentiable, and $L_j$-smooth over~$\X$, i.e., 
\begin{equation}\label{ineqs:convgrad-lipschitz}
F_j(x)+\inner{\nabla F_j(x)}{y-x} \leq F_j(y) \leq F_j(x)+\inner{\nabla F_j(x)}{y-x}+\frac{L_j}{2}\|y-x\|^2, \quad \forall x,y\in \X,
\end{equation}
where $V$ is an open set containing  $\X$.
  
As is well-known,  the convergence rate of the scalar (i.e., $m=1$) Frank-Wolfe (FW) method is $\mathcal{O}(1/k)$, in terms of the sequence of functional
values. Generally, this rate cannot be improved, even with the sole additional assumption that the objective function is strongly convex; see, for example,  \cite{jaggi} and \cite[Section~4.1]{M-AFW} for further details.
However, an enhanced  convergence rate  for the Frank-Wolfe  algorithm in the scalar context has been established under one of the following assumptions, in addition to  {\bf(A1)}: \\ 
(i) The function $F$ is strongly convex on $\X$, and the optimal solution $x^*$ of \eqref{vectorproblem} is in the relative interior of $\X$; see,  for example, \cite{guelat1986some,survey2022};\\
(ii) The set  $\X$ is $(\alpha,q)$-uniformly convex  for some   $\alpha> 0$ and $q\geq2$, and { $\inf_{x\in \X}\|\nabla F(x)\|_*>0$ } (this is equivalent to requiring that  $\X$ does not  contain any unconstrained  minimum  point of $F$); see,  for example, \cite{kerdreux2020,survey2022,garber2015faster};\\
(iii) The set $\X$ is a  $(\alpha,q)$-uniformly convex for some   $\alpha> 0$ and $q\geq2$, and $F$ is strongly convex on $\X$; see,  for example, \cite{kerdreux2020,survey2022,garber2015faster}.

As noted in  \cite[Remark~3]{Leandro2022}, the M-FW   algorithm maintains, under {\bf(A1)},  a convergence rate of $\mathcal{O}(1/k)$. 
The same order of convergence rate was also obtained in \cite{pedro2023,ellen2024}.

The main goal of this paper is to explore, in the multiobjective setting, whether a faster convergence rate can be achieved for the Frank-Wolfe algorithm  under natural generalizations of  assumptions (i), (ii), and (iii) above.  In terms of the objective function values,
 we initially show  that the algorithm achieves a linear convergence rate when the objective function is strongly convex and the limit point of the sequence generated by the method lies in the relative interior of the feasible set.
Next, we turn our attention to a specific class of problems characterized by a feasible constraint set that is $(\alpha,q)$-uniformly convex for some $\alpha > 0$ and $q \geq 2$, notably including \(\ell_p\)-balls for all $p>1$. 
In this scenario, we prove that the method attains:  (i) a  rate of $\mathcal{O}(1/k^\frac{q}{q-1})$ when the objective function is strongly convex; and  (ii) a linear rate (if $q=2$) or a rate of $\mathcal{O}(1/k^{\frac{q}{q-2}})$ (if $q>2$) under an additional assumption, which always holds if the feasible set does not contain an unconstrained weak Pareto point; see Remark~\ref{rem:thetatildexk}. The assumptions and convergence rate results are summarized  in Table \ref{tab:minha_tabela}. We also discuss enhanced convergence rates for M-FW in terms of an optimality measure and present some illustrative examples to showcase our theoretical findings.


\begin{table}[h!]
{\footnotesize
\centering
\begin{tabular}{|c|c|c|c|>{\arraybackslash}m{3.7cm}|} 
\hline
Theorem & Feasible set $\X$ &  Function $F$& limit point  $x^*$ & Conv. rate  \\ \hline
 \cite[Remark~3]{Leandro2022} & convex & convex & --& $\mathcal{O}(1/k)$\\ \hline
 Theorem~\ref{theo:interior} & convex & strongly convex & lies in $\text{relint}(\X)$&  linear  \\ \hline
  Theorem~\ref{uniformly convex2}&  $(\alpha,q)$-uniformly convex  & strongly convex & -- & $\mathcal{O}\left(1/k^\frac{q}{q-1}\right)$ \\ \hline 
 Theorem~\ref{theorem:Xunifconv}&  $(\alpha,q)$-uniformly convex  & $0< \inf_{_{x\in \X}} |\tilde\theta(x)|$& -- & $  \begin{cases}
     \text{linear }&  \text{if } q=2, \\
        \mathcal{O}\left(1/k^{\frac{q}{q-2}}\right) & \text{if } q> 2
    \end{cases} $\\ \hline
\end{tabular}
\\[2mm]
Convergence rates for the multiobjective Frank-Wolfe algorithm under different assumptions. In this table,  $x^*$ is the limit point  of the sequence generated by the algorithm, $\text{relint}(\X)$ denotes the relative interior of $\X$, and $\tilde\theta(x):=\min_{\|z\|\leq 1}\max_{{j \in \J}}\inner{\nabla F_j(x)}{z}$. Since {\bf (A1)} is a common assumption across all results, it is not included in the third column for brevity.
\label{tab:minha_tabela}}
\end{table}


The paper is organized as follows: Section~\ref{sec:prelim}  contains some notations,  definitions and basic results.  Section~\ref{sec:M-FW} formally describes the M-FW algorithm to solve \eqref{vectorproblem} and presents some  basic properties about its generated sequence.  Section~\ref{sec:main} establishes faster convergence rates  for the M-FW algorithm under different assumptions. 
 Section \ref{sec:experiments} contains some illustrative examples.  
 Final remarks are given in Section~\ref{sec:final}.

\section{Preliminary material}\label{sec:prelim}

In this section, we introduce some notations and definitions  which will be used throughout this paper.  Some basic properties on the  multiobjective Frank-Wolfe direction are also discussed.

Let $\J:=\{1,\ldots,m\}$ and $L:=\max_{j \in \J}L_j$. Denote by $B(x,\epsilon) = \{y \in \Rn \ : \ \| x - y \| \leq \epsilon \}$ the ball centered at $x \in \Rn$ with radius $\epsilon > 0$. The diameter of $\X$,  the affine hull of $\X$, and the relative interior of $\X$ are denoted, respectively, by 
\[D_{\X}:=\max\{\|x-y\|: x, y \in \X\}, \quad 
\text{aff}(\X) = \left\{ \sum_{i=1}^n \lambda_i x_i : x_i \in \X, \lambda_i \in \mathbb{R}, \sum_{i=1}^n \lambda_i = 1 \right\}
\] 
and
\[ 
\text{relint}(\X) = \{ x \in \X : \text{there exists } \epsilon > 0 \text{ such that } B(x, \epsilon) \cap \text{aff}(\X) \subseteq \X \}.
\]

 
Given  a pair $(\alpha, q) \in \R^2$ such that $\alpha> 0$ and $q\geq2$, the set $\X$ is $(\alpha,q)$-uniformly convex
if and only if for any $x, y \in \X$, $\gamma\in [0,1]$, and $z\in \R^n$ with $\|z\|\leq 1 $, the following inclusion holds
\begin{equation}\label{def:alpha-q unifconv}
  x+\gamma(y-x)+\gamma(1-\gamma)\frac{\alpha}{2}\|y-x\|^qz\in \X.  
\end{equation}

Uniformly convex sets encompass \(\ell_p\)-balls, which are \(\left({p - 1}, 2\right)\)-uniformly convex for \(1 < p \leq 2\) and \(\left({2}/{p}, p\right)\)-uniformly convex for \(p > 2\), in the context of the \(\|\cdot\|_p\) norm. An \(\left({\alpha}, 2\right)\)-uniformly convex set is also called  \(\alpha\)-strongly convex set. Notably, for \(1 < p \leq 2\), the \(\ell_p\)-ball is \(2(p - 1)\)-strongly convex with respect to the \(\|\cdot\|_p\) norm. We refer the reader to  \cite{kerdreux2020, survey2022} for more details about the concept of $(\alpha,q)$-uniformly convex.

In the multiobjective optimization setting, the concept of optimality is  replaced by the concept of {\it Pareto-optimality} or {\it efficiency}. A point $x^{\ast} \in \X$ is called {\it Pareto optimal} or {\it efficient} if and only if there is no $x \in \X$ such that $ F(x) \leq F(x^{\ast})$ and $F(x)\neq F(x^{\ast})$, where the 
 partial order  $\leq$ in $ \R^m $ is defined componentwise, i.e., $z\leq  y$   if and only if  
$z_j\leq y_j\;$, for all $j\in \J$. Similarly, we have $z<y$   if and only if 
$z_j< y_j\;$, for all $j\in \J$.   A point $x^{\ast} \in \X$ is called {\it weak Pareto optimal} or {\it weakly efficient} if and only if there is no $x \in \X$ such that $ F(x) < F(x^{\ast})$.  When the objectives and the feasible set are convex, a necessary and sufficient condition for a point  $x^* \in \X$  be weak Pareto optimal for \eqref{vectorproblem}  is   that $x^{\ast}$ is {\it Pareto critical} ({\it stationary}), i.e., 
\begin{equation}\label{critical1}
 \max_{j \in \J}\inner{\nabla F_j(x^{\ast})}{y-x^{\ast}} \geq 0, \quad \forall y \in \X.
\end{equation}

In the following, we recall a key result that will be used to establish the convergence rate of the M-FW algorithm when the constraint set $\X$ is $(\alpha,q)-$uniformly convex.
\begin{lemma}[Lemma 2.21 of \cite{survey2022}]\label{lemmarecur} Let  $(\hat h_k)_{k \in \N}$  be a sequence of positive numbers and $c_0$, $c_1$, $c_2$, $\beta$ be positive numbers with $c_1 < 1$ such that $\hat h_1 \leq c_0$ and 
$  \hat h_{k+1} \leq [1-\min\{c_1, c_2 \hat h_k^\beta\}] \hat h_k $ for $k \geq 1$, then
\begin{align*}
    \hat h_k &\leq
    \begin{cases}
        c_0(1 - c_1)^{k-1} & \text{for } 1 \leq k \leq k_0, \\
        \left( \frac{(c_1/c_2)}{1+c_1\beta(k-k_0)} \right)^{\frac{1}{\beta}}=\mathcal{O}(1/k^{1/\beta}) & \text{for } k \geq k_0,
    \end{cases}
\end{align*}
where
\[
k_0 := \max\left\{ \left\lfloor \log_{1-c_1} \left( \frac{(c_1/c_2)^{\frac{1}{\beta}}}{c_0} \right) \right\rfloor + 2, 1 \right\}.
\]
\end{lemma}
  \section{Multiobjective Frank-Wolfe algorithm and basic properties}\label{sec:M-FW}
In this section, we state the multiobjective Frank-Wolfe algorithm and present some of its basic properties, which will be  useful to analyze its convergence rates.

We begin by stating the algorithm and recalling some properties regarding  the optimal value  of its subproblem.

\mgap
\mgap
\noindent
\fbox{
\begin{minipage}[h]{6.4 in}
{\bf Multiobjective Frank-Wolfe (M-FW) algorithm.} 
\begin{itemize}
\item[(0)]  Let $x^0 \in \X$ be given,  and set $L:=\max_{j \in \J}L_j$ and   $k \leftarrow 0$.

\item[(1)] Compute 
\begin{equation}\label{probgrad1}
p^{FW}(x^k)\in \argmin_{y\in \X} \max_{j \in \J }\inner{\nabla F_j(x^k)}{y-x^k}
\end{equation}
and let $d^{FW}(x^k):=p^{FW}(x^k)-x^k$ and  $  \theta^{FW}(x^k):=\max_{j \in \J }\inner{\nabla F_j(x^k)}{d^{FW}(x^k)}$.

\item [(2)] If  $  \theta^{FW}(x^k)=0$, then stop and return $x^k$.
       
\item[(3)] Compute the stepsize $\gamma_k \in (0,1]$ defined as
\begin{equation}\label{def:gamma-k}
     \gamma_k:=\min\left\{1, \frac{-\theta^{FW}(x^k)}{L\|d^{FW}(x^k)\|^2}\right\},
\end{equation}
and set $x^{k+1}:=x^k+\gamma_kd^{FW}(x^k)$ and $k\leftarrow k+1,$ and go to step{ (1)}.
\end{itemize}
\noindent
{\bf end}
\end{minipage}
}
\vgap
\vgap

\begin{remark}\label{rem:M-FW alg}
(i) Since $\X$ is a compact set, it follows that \eqref{probgrad1} has always a solution. Note that problem~\eqref{probgrad1} can be reformulated as
\begin{equation} \label{vproblem}
 \begin{array}{cl}
\ds\min   &  t          \\
\mbox{s. t.} & \inner{\nabla F_j(x)}{y-x} \leq t, \quad \forall j=1,\ldots,m,\\
&t \in \R, \quad y \in \X. 
\end{array}
\end{equation}
which is a constrained convex optimization problem. When $\X$ is a polytope, it becomes a linear programming problem.  (ii) The optimum value function $\theta^{FW}(\cdot)$  has interesting properties and, in particular, it can  characterize stationary points of \eqref{vectorproblem}. Indeed, the following properties hold, see \cite{Leandro2022} for their proofs: (a) $\theta^{FW}(x)\leq0$ for all $x\in \X$; (b) $\theta^{FW}(\cdot)$ is a continuous function on $\X$; (c) $\theta^{FW}(x)=0$ iff  $x\in \X$ is a stationary point.
(iii) It is immediate to see that  if the M-FW algorithm does not stop in step~(2), then  $\theta^{FW}(x^k)< 0$ and, in particular, 
$d^{FW}(x^k)\neq 0$. Therefore, $\gamma_k$ is well-defined and belongs to $(0,1]$. Moreover, as $x^0\in \X$, it follows from the definitions of $d^{FW}(x^k)$ and $x^{k+1}$, and the convexity of $\X$,  that the whole sequence $(x^k)_{k\in \N}$ is contained in $\X$. 
\end{remark}

In the following, we present some  results that will be useful to establish the convergence rates of the M-FW algorithm. The first result shows the progress, in terms of objective value, between two consecutive iterates generated by the algorithm.

\begin{lemma}\label{seq1:ck} 
Let $(x^k)_{k \in \N}$  be the sequence generated by  M-FW algorithm. Then,  for every $j\in \J$, 
it holds:
\begin{equation}\label{key1}
F_j(x^{k+1}) \leq F_j(x^{k})+\frac{\theta^{FW}(x^k)}2\gamma_k , \quad \forall  k \in \N.
\end{equation}
\end{lemma}
\begin{proof}  
Since  $F_j$ is   $L_j$-smooth  on $\X$  and $L=\max_{j  \in \J}L_j$, it follows from the second inequality in \eqref{ineqs:convgrad-lipschitz} that, for all $k \in \N$ and $j \in \J$,
\begin{align}
F_j(x^{k+1})=F_j(x^k+\gamma_k d^{FW}(x^k)) \nonumber
&\leq F_j(x^{k})+\gamma_k \inner{\nabla F_j(x^k)}{d^{FW}(x^k)}+\frac {L_j}2\gamma_k^2\|d^{FW}(x^k)\|^2\\  \nonumber
&\leq F_j(x^{k})+\gamma_k \max_{j \in \J}\inner{\nabla F_j(x^k)}{d^{FW}(x^k)}+\frac {L}2\gamma_k^2\|d^{FW}(x^k)\|^2\\\label{er56}
& = F_j(x^{k})+\gamma_k \theta^{FW}(x^k)+\frac {L}2\gamma_k^2\|d^{FW}(x^k)\|^2,
\end{align}
where the last equality is due to the definition of $ \theta^{FW}(x^k)$ in step~1 of the M-FW algorithm. If $\gamma_k=-\theta^{FW}(x^k)/(L\|d^{FW}(x^k)\|^2)$, it follows from the last inequality that
\begin{align}\label{key12345}
F_j(x^{k+1}) \leq F_j(x^{k})-\frac{\theta^{FW}(x^k)^2}{2L\|d^{FW}(x^k)\|^2},  \quad \forall  k \in \N, j\in \J.
\end{align} 
Now, if  $\gamma_k=1$,  then $L\|d^{FW}(x^k)\|^2\leq -\theta^{FW}(x^k)$. Hence, from \eqref{er56}, we obtain
\begin{align*}
F_j(x^{k+1}) \leq F_j(x^{k})+\frac{\theta^{FW}(x^k)}{2},  \quad \forall k \in \N,  j\in \J.
\end{align*}
 Therefore,  \eqref{key1} follows  from the last inequality, \eqref{key12345} and the definition of $\gamma_k$  in~\eqref{def:gamma-k}.
\end{proof}

It follows from step~(2) of the M-FW algorithm and statement (c) in Remark~\ref{rem:M-FW alg}(ii) that the algorithm stops at some iteration~$k$ if and only if $x^k$ is a weak Pareto point for problem \eqref{vectorproblem}. Hence, from now on, we will assume that the M-FW algorithm generates an infinite sequence $(x^k)_{k\in \N}$.  Since $F$ is a convex function,  we have that every  limit point $x^*$ of $(x^k)_{k \in \N}$  is a  weak Pareto point  of \eqref{vectorproblem};  {see \cite[Remark~2]{Leandro2022}}. Moreover, since $\theta^{FW}(x^k)< 0$,  it follows from \eqref{key1} that the sequence $(F_j(x^k))_{k\in \N}$ is monotonically decreasing for all $j\in \J$ and bounded, because $F_j$ is continuous and $\X$ is compact, hence, it converges to some $F^*_j$. Thus, for every $j\in \J$, we have $F_j(x^*)=F^*_j$ for all limit point $x^*$ of $(x^k)_{k \in \N}$.
{Furthermore, under the additional assumption that $F_j$ is $\mu_j$-strongly convex for all $j\in \J$,  
the whole sequence $(x^k)_{k \in \N}$ converges to $x^*$, see the proof of this fact immediately after assumption ${\bf (A2)}$.} 
Therefore, henceforth we fix a limit point $x^*$ of $(x^k)_{k \in \N}$ and define the following elements:
\begin{equation}\label{defhj}
 e^k:=\frac{x^*-x^k}{\|x^*-x^k\|}, \qquad  h_j(x^{k}):=F_j(x^{k})-F_j(x^*),  \qquad   \hat h(x^{k}):= \min_{{j \in \J }} h_j(x^{k}).
\end{equation}

It is easy to see that the sequence $(\hat h(x^k))_{k\in \N}$ converges to zero. Our goal will be to measure how fast this convergence is, under different set of assumptions on the objective function $F$ and the constraint set $\X$. To this end, it will be interesting to give an upper bound on the quantity $\hat h(x^1)$.  The next result shows that, in fact, an upper bound can be universally established for the whole sequence $(\hat h(x^{k+1}))_{k\in \N}$ in terms of the constant of smoothness of $F$, $L=\max_{j\in \J}{L_j}$, and the diameter of the constraint set $\X$.

\begin{lemma} Let $(x^k)_{k \in \N}$  be the sequence generated by the M-FW algorithm. The following inequality holds:
\begin{equation}\label{boundh}
 \hat h(x^{k+1}) \leq \frac {LD_{\X}^2}2, \quad \forall k\in \N,
\end{equation}
where $L:=\max_{j\in \J} L_j$ and $D_\X$ is the diameter of the constraint set $\X$.
\end{lemma} 
\begin{proof}
It follows from \eqref{er56} and  \eqref{defhj}, for all  $k \in \N$, that 
\begin{align} \nonumber
\hat h(x^{k+1}) & \leq \hat h(x^{k})+ \gamma_k \theta^{FW}(x^k)+\frac {L}2\gamma_k^2\|d^{FW}(x^k)\|^2\\
&  \leq \hat h(x^{k})+  \theta^{FW}(x^k)+\frac {L}2\|d^{FW}(x^k)\|^2, \label{eq:450}
\end{align}
where the last inequality is due to the fact that the stepsize $\gamma_k$ given in \eqref{def:gamma-k}  is the minimizer of the quadratic function $q(\gamma):= \gamma \theta^{FW}(x^k)+\frac {L}2\gamma^2\|d^{FW}(x^k)\|^2$ over the interval $[0,1]$. 
On the other hand, it follows from the definitions of $\theta^{FW}(\cdot)$ and $\hat h(x^k)$ in  step~(1) of the M-FW algorithm and \eqref{defhj}, respectively, and the gradient inequality of the convex function $F_j$ (see the first inequality in \eqref{ineqs:convgrad-lipschitz})   that 
\[\theta^{FW}(x^k)\leq \max_{j \in \J }\inner{\nabla F_j(x^k)}{x^*-x^k} \leq \max_{j\in \J }(F_j(x^*)-F_j(x^k))=  -\min_{j \in \J }h_j(x^k)=-\hat h (x^k),\quad \forall  k \in \N.\]
Therefore, \eqref{boundh} follows by combining the last inequality,  \eqref{eq:450}, and the fact that  $\|d^{FW}(x^k)\|\leq D_{\X}$ for all $k\in \N$.
\end{proof}


The next two  results assume that the objective function of problem~\eqref{vectorproblem} is strongly convex. For convenience, this will be stated in the following assumption. 

{\bf(A2)} the function $F$ is strongly convex, i.e.,  the component $F_j\colon \X \rightarrow \mathbb{R}$ of $F$,  for every  $j \in \J$, satisfies
\begin{equation}\label{ineq-Fj stongly conv}
F_j(y)\geq F_j(x)+\inner{\nabla F_j(x)}{y-x}+\frac{\mu_j}2\|y-x\|^2, \quad x,y\in \X,
\end{equation}
for some  $\mu_j>0$. For simplicity,  we denote $\mu:=\min_{j \in \J}\mu_j$. 

It is worth noting that, under assumption {\bf (A2)}, the whole sequence $(x^k)_{k \in \N}$ converges to the limit point $x^*$.  Indeed, since $F(x^*)\leq F(x^k)$ for all $k\in \N$, it follows from {\bf (A2)} and the definition of $\theta^{FW}(x^k)$ that 
\begin{align*}
0&\geq \max_{j\in \J} (F_j(x^*)-F_j(x^k))\geq \max_{j\in \J} \langle \nabla F_j(x^k),x^*-x^k\rangle + \frac{\min_{j\in \J}\mu_j}{2}\|x^k-x^*\|^2 \\
&\geq \theta^{FW}(x^k)+\frac{\mu}{2}\|x^k-x^*\|^2,
\end{align*}
which implies that 
\begin{equation}\label{ineq:theta-xk}
 |\theta^{FW}(x^k)|\geq   \frac{\mu}{2}\|x^k-x^*\|^2, \quad \forall k\in \N.
\end{equation}
Since $(\theta^{FW}(x^k))_{k\in \N}$ converges to zero {(see \cite[Corollary~14]{Leandro2022}}), we conclude from \eqref{ineq:theta-xk} that  the whole sequence $(x^k)_{k \in \N}$ converges to  $x^*$.

In the following, we show that the decreasing of the sequence $(h_j(x^k))_{k\in \N}$ can be, at iteration~$k$, controlled by the angle between  $\nabla F_j(x^k)$ and the direction $e^k$ defined in \eqref{defhj}. The proof of this property is similar to the one in \cite[Lemma~5]{M-AFW}, established for a special variant of the M-FW algorithm, we have included it here for completeness.

\begin{lemma}\label{seq1:ck22} 
Let $(x^k)_{k \in \N}$  be the generated by the M-FW algorithm and consider $(h_j(x^{k}))_{k\in \N}$ and $(e^k)_{k\in \N}$ as in \eqref{defhj}. If {\bf(A2)} holds, then, for every $j\in \J$, we have 
\begin{equation}\label{key2}
0<  h_j(x^{k}) \leq \frac{\inner{\nabla F_j(x^k)}{ e^k}^2}{2\mu}, \quad \forall  k \in \N.
\end{equation}
\end{lemma}
\begin{proof}  The first inequality in \eqref{key2} follows from the definition of $h_j(x^k)$ in \eqref{defhj}, the fact that $F_j(x^k)\geq F_j(x^*)$, and the assumption  that $x^k$ is not a weak Pareto point for \eqref{vectorproblem}. 
Now, since  $F_j$ is $\mu_j$-strongly convex  and $\mu=\min_{j \in \J}\mu_j$, we have, for all $k \in \N$, $j \in\J$ and $\gamma \in [0,1]$, that
\begin{align*}
 F_j(x^k+\gamma (x^*-x^k)) &\geq F_j(x^{k})+\gamma \inner{\nabla F_j(x^k)}{x^*-x^k}+\frac {\mu}2\gamma^2\|x^*-x^k\|^2=:q(\gamma).
\end{align*}
Using that the unconstrained minimizer  of $q(\gamma)$ is $\gamma^*=-\inner{\nabla F_j(x^k)}{x^*-x^k}/(\mu\|x^*-x^k\|^2)$, we obtain, for all $k \in \N$, $j \in \J$, and $\gamma \in [0,1]$,  that
\begin{align*}
 F_j(x^k+\gamma(x^*-x^k)) &\geq F_j(x^{k})- \frac{\inner{\nabla F_j(x^k)}{x^*-x^k}^2}{2\mu\|x^*-x^k\|^2}.
\end{align*}
Taking $\gamma=1$ in the last inequality and using the definition of $e^k$ in \eqref{defhj}, we have, for all $k \in \N$ and $j\in \J$,
\begin{align*}
 F_j(x^*) &\geq F_j(x^{k})- \frac{\inner{\nabla F_j(x^k)}{ e^k}^2}{2\mu},  
\end{align*}
which, combined with the definition of $h_j(\cdot)$ in \eqref{defhj}, proves  the  second inequality in \eqref{key2}.
\end{proof}

We now prove a recursive inequality for the sequence $(\hat h(x^{k}))_{k \in \N}$ that will be essential to establish the convergence rate of the M-FW algorithm under strongly convexity assumption on the objective function $F$. 
\begin{lemma}\label{seq1:ck223} 
Let $(x^k)_{k \in \N}$  be the sequence generated by  the M-FW algorithm, and let $\hat h(x^k)_{k \in \N}$ and $(e^k)_{k \in \N}$ be as defined in \eqref{defhj}. Assume that {\bf(A2)} holds. Then, 
\begin{equation}\label{key123}
\hat h(x^{k+1})\leq \left[1 + \frac{\mu\gamma_k\theta^{FW}(x^k)}{\min_{j \in \J }\inner{\nabla F_j(x^k)}{ e^k}^2} \right]\hat h(x^{k}), \quad \forall k \in \N.
\end{equation}
\end{lemma}
\begin{proof}  
For every $k\in \N$, it follows from Lemma~\ref{seq1:ck22} and  the definition of $\hat h$ in \eqref{defhj} that
\[
0<\hat h(x^{k})\leq \frac{\min_{j\in \J}\inner{\nabla F_j(x^k)}{ e^k}^2}{2\mu},
\]
which yields
\[
\frac{1}{2\hat h(x^{k})}\geq  \frac{\mu}{\min_{j \in \J}\inner{\nabla F_j(x^k)}{ e^k}^2}.
\]
On the other hand, it follows from  \eqref{key1} and the definition of $\hat h$ in \eqref{defhj} that
\begin{align*}
\hat h(x^{k+1})&\leq \left[1 +\frac{\gamma_k\theta^{FW}(x^k)}{2\hat h(x^{k})} \right]\hat h(x^{k}).
\end{align*}
Therefore, \eqref{key123} follows by combining  the last two inequalities  and the facts that $\theta^{FW}(x^k)<0$ and $\hat h(x^k)>0$.
\end{proof} 

The next section is devoted to the main results of this paper. Specifically, under different set of assumptions on the objective function $F$ and the feasible set $\X$, we establish improved convergence rates for the  sequences  $(\hat h(x^k))_{k\in \N}$ and $(\theta(x^k))_{k\in \N}$.

\section{Improved convergence rates for the M-FW algorithm}\label{sec:main}

We start this section by establishing the linear convergence of the  M-FW algorithm under the  assumptions, besides of {\bf (A1)}, that (i) $F$ is strongly convex on $\X$  and (ii) the  limit point  $x^*$ of  $(x^k)_{k \in \N}$ is in the relative interior of $\X$.


\begin{theorem}\label{theo:interior}
Assume that  {\bf(A2)} holds and   $x^* \in \text{relint}(\X)$. Let $r>0$ and $k_0\in \N$ be such that  $B(x^*, 2r)\cap \text{aff}(\X)\subset \X$ and $x_k\in B(x^*, r)$ for all $k\geq k_0$. Then, the following inequalities hold
\begin{equation}\label{ineq1:convLinear}
\hat h(x^{k+1})\leq {\left[1 - \frac{\mu r^2}{LD_{\X}^2}\right]\hat h(x^{k})}, \quad \forall k \geq k_0,
\end{equation}
and
\begin{equation}\label{ineq1:convRLinear-xktheta}
\sum_{i=m}^k|\theta^{FW}(x^i)|^2\leq 2L\D^2_\X \hat h(x^m), \quad \forall  k\geq m \geq 1.
\end{equation}
\end{theorem}

\begin{proof}  First note that the existence of $r>0$ and $k_0\in \N$ follows  from the fact that the sequence $(x^k)_{k \in \mathbb{N}}$ converges to $x^* \in \text{relint}(\X)$.
Hence, for all $k\geq k_0$, we trivially have   $y^k:=x^k+r  e^k \in \X$, where $ e^k=(x^*-x^k)/\|x^*-x^k\|$. Therefore,  it follows   from \eqref{probgrad1} that
\[
\theta^{FW}(x^k)= \min_{y\in \X} \max_{j \in \J }\inner{\nabla F_j(x^k)}{y-x^k}\leq{r}\max_{j \in \J }\inner{\nabla F_j(x^k)}{ e^k}
<0, \]
where the last inequality is due to the fact that, for all $j \in \J$,  $F_j$ is convex and $F_j(x^*)-F_j(x^k)=-h_j(x^k)<0$ (see \eqref{key2}). Thus,
\begin{equation}\label{a89r}
 -\left(\theta^{FW}(x^k)\right)^2
\leq-r^2\left( \max_{j \in \J }\inner{\nabla F_j(x^k)}{ e^k}\right)^2=-r^2 \min_{j \in \J }\inner{\nabla F_j(x^k)}{ e^k}^2.   
\end{equation}
On the other hand, we have  that \( \gamma_k < 1 \) for all \( k \geq k_0 \). Indeed,  if \( \gamma_k = 1 \), then \( p^{FW}(x^k)=x^{k+1}  \in B(x^*, r) \), which implies that $p^{FW}(x^k)\in  \text{relint}(\X)$. Hence, $z^k:=x^k+t(p^{FW}(x^k)-x^k)\in \X$ for some $t >1$. Thus, we have 
\[
\theta^{FW}(x^k)= \min_{z\in \X} \max_{j \in \J }\inner{\nabla F_j(x^k)}{z-x^k}\leq\max_{j \in \J }\inner{\nabla F_j(x^k)}{ t(p^{FW}(x^k)-x^k)}=t\theta^{FW}(x^k),\]
which is a contradiction with the fact that $t>1$ and $\theta^{FW}(x^k)<0$.
Hence, 
it follows from \eqref{def:gamma-k} that $\gamma_k= {-\theta^{FW}(x^k)}/{(L\|d^{FW}(x^k)\|^2)}$, which combined with \eqref{key123} yields 
\[ 
\hat h(x^{k+1})\leq \left[1 - \frac{\mu(\theta^{FW}(x^k))^2}{L\|d^{FW}(x^k)\|^2\min_{j \in \J }\inner{\nabla F_j(x^k)}{ e^k}^2} \right]\hat h(x^{k}).
\]
Therefore, \eqref{ineq1:convLinear} follows from the latter inequality, \eqref{a89r} and the fact that $\|d^{FW}(x^k)\|\leq D_{\X}$. Now, to prove \eqref{ineq1:convRLinear-xktheta},  note that  \eqref{key1} and the fact that $\gamma_k= {-\theta^{FW}(x^k)}/{(L\|d^{FW}(x^k)\|^2)}$, imply that
$$
    \frac{|\theta^{FW}(x^k)|^2}{2L\|d^{FW}(x^k)\|^2} \leq F_j(x^{k}) - F_j(x^{k+1}), \quad \forall  k \in \N,\; \forall j\in \J,
$$
which, combined with the fact that $\|d^{FW}(x^k)\|\leq \D_\X$, yields, for all $ j\in \J$ and $k\geq m\geq 1$,

$$
   \sum_{i=m}^k |\theta^{FW}(x^i)|^2\leq 2L\D_\X^2  \sum_{i=m}^k(F_j(x^{i}) - F_j(x^{i+1})) = 2L\D_\X^2 (F_j(x^{m}) - F_j(x^{k+1})).
$$
Therefore,  \eqref{ineq1:convRLinear-xktheta} follows
from the last inequality,  the fact that $F(x^{k+1})\geq F(x^*)$ and the definition of $\hat h(x^m)$ in \eqref{defhj}.
\end{proof}

 \begin{remark}\label{Rem:def-thetakbest}
 It follows from \eqref{ineq:theta-xk}  and \eqref{ineq1:convRLinear-xktheta} with $m=k$ that \[ \frac{\mu}{2}\|x^{k+1}-x^*\|^2\leq |\theta^{FW}(x^k)| \leq \sqrt{2L}\D_\X \sqrt{\hat h(x^k)}, \quad k \geq 1. \] On the other hand,  by recursively using \eqref{ineq1:convLinear}, we find that  
 \begin{equation}\label{a5t6}
 \hat h(x^{k})\leq \left[1 - \frac{\mu r^2}{{LD_{\X}^2}}\right]^{k-k_0}\hat h(x^{k_0}), \quad \forall k\geq k_0,
 \end{equation}
 where $k_0$ is as in Theorem~\ref{theo:interior}. From \eqref{boundh}, we immediately have
 \begin{equation}\label{3er5}
   \sqrt{\hat h(x^{k_0})}\leq \sqrt{L}\D_\X/\sqrt{2}.   
 \end{equation}
 Thus, combining the last three inequalities, we get 
 $$
 \frac{\mu}{2}\|x^{k+1}-x^*\|^2 \leq |\theta^{FW}(x^k)|\leq {L{\cal D}^2_{\X}}\left[\sqrt{1 - \frac{\mu r^2}{L{\cal D}_{\X}^2}}\right]^{k-k_0}, \quad \forall k \geq k_0.
 $$
 The above inequalities imply that the sequences $(x^k)_{k\in \N}$ and $(\theta^{FW}(x^k))_{k\in \N}$ converge $R$-linearly under the assumptions of Theorem~\ref{theo:interior}. Moreover, it follows from  \eqref{ineq1:convRLinear-xktheta} with $m=\lfloor k/2\rfloor$,  \eqref{a5t6} and \eqref{3er5} that 
 $$\theta^k_{best}\leq \frac{\sqrt{2}L{\cal D}_{\X}^2}{\sqrt{k}}\left[\sqrt{1 - \frac{\mu r^2}{L{\cal D}_{\X}^2}}\right]^{\lfloor k/2\rfloor-k_0}, \forall k>2k_0,$$
 where $\theta^k_{best}:=\min_{i\in \{\lfloor k/2\rfloor, \ldots,k\}}|(\theta^{FW}(x^i))|$. 
 Roughly speaking, for $k$ sufficiently large, the above bound is of ${\cal O}(\beta^k/\sqrt{k} )$ for some $\beta\in (0,1)$. 
  Recall that {\cite[Theorem~16]{Leandro2022}}  established, under {\bf (A1)}, a convergence rate of  ${\cal O}(1/{k})$ for the sequence  $(\theta^k_{best})_{k\in \N}$.
Therefore, under the additional hypothesis of Theorem~\ref{theo:interior}, we have shown an improved convergence rate for  $(\theta^k_{best})_{k\in \N}$ as well as new convergence rates for $(\theta^{FW}(x^k))_{k\in \N}$ and $(x^k)_{k\in \N}$.
\end{remark}
 
The next two theorems are devoted to the analysis of the convergence rate of the M-FW algorithm under the additional assumption that the feasible set $\X$ is uniformly convex. Both results use the fundamental  recursive inequality on the sequence $(\hat h(x^k))_{k\in \N}$  presented in the following lemma. It is interesting to note that this lemma  does not require  the objective function $F$ to be strongly convex.

\begin{lemma}\label{lemma:Xunifconv} Assume   that the constraint set $\X$ is  $(\alpha,q)$-uniformly convex for some   $\alpha> 0$ and $q\geq2$. 
Let $(x^k)_{k \in \N}$  be the sequence generated by  the M-FW algorithm. 
Then, the following inequalities hold
\begin{equation}\label{ineq:basicRecursive-thetatilde}
 |\theta^{FW}(x^k)| \geq    \frac\alpha4 \|d^{FW}(x^k)\|^q |\tilde \theta(x^k)|, \qquad \hat h(x^{k+1}) \leq \left[1-\min\left\{\frac12,r_k \right\}\right]\hat h(x^{k}), \quad \forall k \in \N, 
\end{equation}
where
\begin{equation}\label{def:rk}
r_k:= \frac{[\hat h(x^{k})]^{(q-2)/q}[\alpha |\tilde \theta(x^k)|]^{2/q}}{2^{(4+q)/q}L}, \qquad  \tilde \theta(x):=\min_{\|z\|\leq 1}\max_{{j \in \J}}\inner{\nabla F_j(x)}{z}.
\end{equation}
\end{lemma}
\begin{proof}
Let $k\in \N$. It follows from  step~1 of the M-FW algorithm, the first inequality in \eqref{ineqs:convgrad-lipschitz} and  \eqref{defhj} that 
\begin{align}\label{ineq:aux-a1}
\theta^{FW}(x^k)=\min_{y\in\X}\max_{j \in \J }\inner{\nabla F_j(x^k)}{y-x^k}&\leq \max_{j \in \J  }\inner{\nabla F_j(x^k)}{ x^*-x^k}\nonumber\\\nonumber
&\leq \max_{j \in \J }F_j(x^*)-F_j(x^k)\\
&= -\min_{j \in \J  }h_j(x^k)=-\hat h (x^k).
\end{align}
Now,  let $z^k$ and $\tilde x^k$ be defined as
\[
 z^k:=\argmin_{\|z\|\leq 1}\max_{{j \in \J }}\inner{\nabla F_j(x^k)}{z}, \quad  \tilde x^k:=x^k+\frac12d^{FW}(x^k)+\frac\alpha8 \|d^{FW}(x^k)\|^qz^k.
\]
Since $d^{FW}(x^k)= p^{FW}(x^k)-x^k$, we have  $\tilde x^k\in \X$ in view of \eqref{def:alpha-q unifconv} with $x=x^k, y=p^{FW}(x^k), z=z^k$ and $\gamma=1/2$. 
Hence,  it follows from the definitions of $\theta^{FW}(x^k)$ and $\tilde \theta(x^k)$  that
\begin{align*}
\theta^{FW}(x^k)&=\min_{y\in\X}\max_{j \in \J }\inner{\nabla F_j(x^k)}{y-x^k}\leq \max_{j \in \J }\inner{\nabla F_j(x^k)}{\tilde x^k-x^k}\\
&\leq \frac12\max_{j \in \J  }\inner{\nabla F_j(x^k)}{d^{FW}(x^k)} +\frac\alpha8 \|d^{FW}(x^k)\|^q   \max_{j \in \J  }\inner{\nabla F_j(x^k)}{z^k}\\
&= \frac12\theta^{FW}(x^k)+\frac\alpha8 \|d^{FW}(x^k)\|^q \tilde \theta(x^k),
\end{align*}
which, combined with the fact that $\max\{\theta^{FW}(x^k),\tilde\theta(x^k)\} <0$,   implies the first inequality in \eqref{ineq:basicRecursive-thetatilde}.
Now,  using \eqref{key1} and \eqref{ineq:aux-a1} and the definition of $\hat h$ in \eqref{defhj}, we obtain
\begin{align}\label{ineq:recursive-hj}
\hat h(x^{k+1}) 
& \leq \hat h(x^{k})+\frac{\gamma_k}{2} \theta^{FW}(x^k)\nonumber\\
&\leq \left(1-\frac{\gamma_k}2\right)\hat h(x^{k}).
\end{align}
It follows from the definition of $\gamma_k$ in \eqref{def:gamma-k},   \eqref{ineq:aux-a1}, the first inequality in \eqref{ineq:basicRecursive-thetatilde}, and the fact that $\max\{\theta^{FW}(x^k),\tilde\theta(x^k)\} <0$,  that
\begin{align*}
\gamma_k&=\min\left\{1, \frac{|\theta^{FW}(x^k)|}{L\|d^{FW}(x^k)\|^2}\right\}= \min\left\{1, \frac{|\theta^{FW}(x^k)|^{1-2/q}|\theta^{FW}(x^k)|^{2/q}}{L\|d^{FW}(x^k)\|^2}\right\}\\
&\geq \min\left\{1, \frac{[\hat h(x^{k})]^{1-2/q}[\frac\alpha4 \|d^{FW}(x^k)\|^q |\tilde \theta(x^k)|]^{2/q}}{L\|d^{FW}(x^k)\|^2}\right\}\\
&= \min\left\{1, \frac{[\hat h(x^{k})]^{1-2/q}[\alpha |\tilde \theta(x^k)|]^{2/q}}{2^{4/q}L}\right\}
\end{align*}
which, combined with \eqref{ineq:recursive-hj} and the definition of $r_k$ in \eqref{def:rk},   proves the second inequality in \eqref{ineq:basicRecursive-thetatilde}.
\end{proof}

We next establish a faster convergence rate  for the M-FW algorithm by assuming, besides {\bf (A1)}, that {\bf (A2)} holds and the set $\X$ is $(\alpha,q)$-uniformly convex. Roughly speaking,  we show that $\hat h(x^k)=\mathcal{O}(1/k^{q/(q-1)})$. In particular, if  $\X$ is strongly convex, i.e., $(\alpha,2)-$uniformly convex, then  $\hat h(x^k)=\mathcal{O}(1/k^{2})$, which improves the sublinear convergence rate $\hat h(x^k)=\mathcal{O}(1/k)$ proved in \cite{Leandro2022} under the assumption of convexity of $F$.

\begin{theorem}\label{uniformly convex2}
Let $(x^k)_{k \in \N}$  be the sequence generated by  the M-FW  algorithm.  Assume that the constraint set $\X$  is $(\alpha,q)$-uniformly convex for some   $\alpha> 0$ and $q\geq2$, and that {\bf(A2)} holds. Then,  the following inequalities hold
\begin{align}\label{ineq:convrate-hk2}
    \hat h(x^k) &\leq
    \begin{cases}
        \frac{LD_{\X}^2}{2^k}, & \text{for }  1 \leq k \leq k_0, \\
        \left(\frac{8}{\alpha^2\mu}\right)^{1/(q-1)}\left(\frac{2qL}{2q+(q-1)(k-k_0)} \right)^{\frac{q}{q-1}}=\mathcal{O}\left(1/k^\frac{q}{q-1}\right), & \text{for } k > k_0,
    \end{cases}
\end{align}
and
\begin{equation}\label{convrate-thetak-2}
     \sum_{i=m}^k\min\left\{|\theta^{FW}(x^i)|,\frac{|\theta^{FW}(x^i)|^2}{L\|d^{FW}(x^i)\|^2}\right\} \leq 2\hat h(x^m), \quad \forall k\geq m\geq 1,
\end{equation}
where 
\[
k_0 := \max\left\{ \left\lfloor \log_{2} \left( \frac{D_\X^2}{2}\left(\frac{\alpha^2\mu}{8L}\right)^{1/(q-1)} \right) \right\rfloor + 2, 1 \right\}.
\]
\end{theorem}
\begin{proof} 
Let  $e^k$ and $\tilde \theta(x^k)$ as in \eqref{defhj}  and \eqref{def:rk}, respectively. Using the first inequality in  \eqref{ineqs:convgrad-lipschitz},  we have
$$
\tilde \theta(x^k)\leq \max_{j\in\J }\inner{\nabla F_j(x^k)}{e^k}\leq \frac{\max_{j\in\J }F_j(x^*)-F_j(x^k)}{\|x^k-x^*\|}<0.
$$
Hence, since $F_j$ is $\mu$-strongly convex for every  $j\in J$, it follows from \eqref{key2} and the definition of $\hat h(x^k)$ in \eqref{defhj} that
\[
0<\hat h(x^{k}) \leq \frac{\min_{j \in \J }\inner{\nabla F_j(x^k)}{ e^k}^2}{2\mu}=\frac{(\max_{j \in \J }\inner{\nabla F_j(x^k)}{ e^k})^2}{2\mu}\leq \frac{\tilde \theta(x^k)^2}{2\mu},
\]
which yields
\[
\tilde \theta(x^k)^2\geq 2\mu\hat h(x^{k}).
\]
It follows by combining the latter inequality  and the definition of $r_k$ in \eqref{def:rk} that
$$
r_k=\frac{[\hat h(x^{k})]^{(q-2)/q}}{2L}\left(\frac{\alpha^2 \tilde \theta(x^k)^2}{16}\right)^{1/q} \geq \frac{[\hat h(x^{k}) ]^{(q-1)/q}}{2L}\left(\frac{\alpha^2\mu}{8}\right)^{1/q}.
$$
Hence,  from the second inequality in \eqref{ineq:basicRecursive-thetatilde}, we have
\begin{align*}
  \hat h(x^{k+1}) \leq \left[1-\min\left\{\frac{1}{2},\frac{[\hat h(x^{k}) ]^{(q-1)/q}}{2L}\left(\frac{\alpha^2\mu}{8}\right)^{1/q} \right\}\right]\hat h(x^{k}), \qquad  \forall k\geq 1.
\end{align*} 
Moreover, it follows from \eqref{boundh} that 
\[ \hat h(x^{1}) \leq \frac {LD_{\X}^2}2.\]
Therefore,  \eqref{ineq:convrate-hk2} follows immediately from the last two inequalities and  Lemma~\ref{lemmarecur} with $$c_0=\frac{LD_\X^2}{2}, \quad c_1=\frac{1}{2}, \quad c_2=\frac{1}{2L}\left(\frac{\alpha^2\mu}{8}\right)^{\frac{1}{q}},  \quad \beta=\frac{q-1}{q},
$$
and some algebraic manipulations.  Now, to prove \eqref{convrate-thetak-2},  note that  \eqref{key1} and the definition of $\gamma_k$ in \eqref{def:gamma-k} imply that
$$
    \frac{1}{2}\min\left\{|\theta^{FW}(x^i)|,\frac{|\theta^{FW}(x^i)|^2}{L\|d^{FW}(x^i)\|^2}\right\} \leq F_j(x^{i}) - F_j(x^{i+1}), \quad \forall  i\in \N,\; \forall j\in \J.
$$
Hence, for all $ j\in \J$ and $k\geq m\geq 1$, we have
\begin{align*}
   \sum_{i=m}^k \min\left\{|\theta^{FW}(x^i)|,\frac{|\theta^{FW}(x^i)|^2}{L\|d^{FW}(x^i)\|^2}\right\} &\leq 2\sum_{i=m}^k(F_j(x^{i}) - F_j(x^{i+1})) \\
   &= 2(F_j(x^{m}) - F_j(x^{k+1})).
\end{align*}
Therefore, \eqref{convrate-thetak-2} follows
from the last inequality,  the fact that $F(x^{k+1})\geq F(x^*)$, and  the definition of $\hat h(x^m)$ in \eqref{defhj}.
\end{proof}
\begin{remark}\label{rem:secremarkonthetak}
It follows from \eqref{convrate-thetak-2} with $m:=k$ and the fact that $\|d^{FW}(x^k)\|\leq \D_\X$ that 
\begin{equation}\label{ineq:aux-rem3}
     \min\left\{|\theta^{FW}(x^k)|,\frac{|\theta^{FW}(x^k)|^2}{L\D_\X^2}\right\} \leq 2\hat h(x^k), \quad \forall k\geq  1.
\end{equation}
In particular,  the sequence $(\theta^{FW}(x^k))_{k\in \N}$ converges to zero and, for all $k$ sufficiently large, we have $|\theta^{FW}(x^k)|^2/[L\D_\X^2]\leq |(\theta^{FW}(x^k))|$. Hence, from \eqref{ineq:theta-xk}, \eqref{ineq:convrate-hk2}, and \eqref{ineq:aux-rem3},  we have, for all $k$ sufficiently large, 
 $$
 \frac{\mu}{2}\|x^{k+1}-x^*\|^2 \leq |\theta^{FW}(x^k)|\leq D_{\X}\sqrt{2L\hat h(x^k)}\approx \mathcal{O}\left(1/k^\frac{q}{2(q-1)}\right),
 $$
which implies that the sequences $(x^k)_{k\in \N}$ and $(\theta^{FW}(x^k))_{k\in \N}$ converge at a rate $\mathcal{O}\left(1/k^\frac{q}{4(q-1)}\right)$ and $\mathcal{O}\left(1/k^\frac{q}{2(q-1)}\right)$, respectively, under the assumptions of Theorem~\ref{uniformly convex2}.
Again as $(\theta^{FW}(x^k))_{k\in \N}$ converges to zero, it follows \eqref{ineq:convrate-hk2}, \eqref{convrate-thetak-2} with  $m=\lfloor k/2\rfloor$ and simple algebraic manipulations that, for all $k$ sufficiently large,
$$\theta^k_{best}\approx \mathcal{O}\left(\frac{1}{k^\frac{2q-1}{2(q-1)}}\right), $$
where $\theta^k_{best}$ is as in Remark~\ref{Rem:def-thetakbest}. Since $q \geq 2$, 
the above bound improves the rate of ${\cal O}(1/k)$ for  the  sequence  $(\theta^k_{best})_{k\in \N}$ established in  {{\cite[Theorem~16]{Leandro2022}}}  under {\bf (A1)}.
\end{remark}

We are now ready to establish a convergence rate result for the M-FW algorithm when (i) the constraint set $\X$ is $(\alpha,q)-$uniformly convex and (ii) the sequence $(\tilde{\theta}(x^k))_{k \in \N}$ defined in \eqref{def:rk} stays away from zero.  The next convergence rate bounds are better than those in Theorem~\ref{uniformly convex2}, and, in particular,  the linear convergence is achieved  when the constraint set is strongly convex (i.e., $(\alpha,2)-$uniformly convex).  Although the assumption on $(\tilde{\theta}(x^k))_{k \in \N}$ is considered on the sequence generated by the M-FW algorithm, we will show in Remark~\ref{rem:thetatildexk} that it holds under suitable conditions on problem~\eqref{vectorproblem}.



\begin{theorem}\label{theorem:Xunifconv} Let $(x^k)_{k \in \N}$  be the sequence generated by  the M-FW algorithm. Assume that  the  constraint set $\X$  is $(\alpha,q)$-uniformly convex for some   $\alpha> 0$ and $q\geq2$,  and   $c:=\inf_k|\tilde \theta(x^k)|>0$. Then,  
\begin{align}\label{rt65}
    \hat h(x^k) &\leq
    \begin{cases} \left[\max\left\{\frac12,1-\frac{\zeta}{L}\right\}\right]\hat h(x^{k-1}) &  \text{for } q=2, k\geq 1,\\
        \frac{LD_{\X}^2}{2^k} & \text{for } q> 2,  1 \leq k \leq k_0, \\
        \left( \frac{qL}{\zeta[2q+(q-2)(k-k_0)]} \right)^{\frac{q}{q-2}}=\mathcal{O}\left(1/k^{\frac{q}{q-2}}\right) & \text{for } q> 2 ,  k > k_0,
    \end{cases}
\end{align}
and
\begin{equation}\label{convrate-thetak-3}
     \sum_{i=m}^k\min\left\{|\theta^{FW}(x^i)|,\frac{|\theta^{FW}(x^i)|^2}{L\|d^{FW}(x^i)\|^2}\right\} \leq 2\hat h(x^m), \quad \forall k\geq m\geq 1,
\end{equation}
where
\[
 \quad \zeta=\zeta(\alpha,q,c):=\frac{[\alpha c]^{2/q}}{{2^{(4+q)/q}}}, \quad 
k_0 := \max\left\{ \left\lfloor \frac{2}{q-2}\log_{2} \left( \frac{2D_{\X}^{q-2}\zeta^{q/2}}{L}\right) \right\rfloor + 2, 1 \right\}.
\]
\end{theorem}
\begin{proof} Since $c=\inf_k|\tilde \theta(x^k)|>0$, it follows from the second inequality in \eqref{ineq:basicRecursive-thetatilde} and the definition of $\zeta$ that 
\begin{equation*}
      \hat h(x^{k+1}) \leq \left[1-\min\left\{\frac{1}{2},\frac{\zeta}{L}[\hat h(x^{k})]^{(q-2)/q}\right\}\right]\hat h(x^{k}), \quad \forall k \in \N. 
\end{equation*} 
The first inequality in \eqref{rt65} follows immediately from the above inequality with $q=2$. Now, note that \eqref{boundh} implies that 
\[ \hat h(x^{1}) \leq \frac {LD_{\X}^2}2.\]
 Thus,  if  $q>2$ the result follows immediately  from the last two  inequalities above, Lemma~\ref{lemmarecur} with $c_0=LD_{\X}^2/2$, $c_1=1/2$,  $c_2={\zeta/L}$ and $\beta=(q-2)/q$, and some algebraic manipulations. 
Since the proof of \eqref{convrate-thetak-3} follows the same steps as those of \eqref{convrate-thetak-2}, it is omitted.
\end{proof}
We finish this section by presenting two remarks. The first one shows how convergence rate of the sequence $(\theta(x^k))_{k\in \N}$ can be obtained from Theorem~\ref{theorem:Xunifconv}, and compares it with the one obtained in \cite{Leandro2022}. The second remark discusses some conditions that imply the assumption on $(\tilde \theta(x^k))_{k\in \N}$ considered in Theorem~\ref{theorem:Xunifconv}.
\begin{remark}
Similarly to Remark~\ref{rem:secremarkonthetak}, we can show that the sequences   $(\theta(x^k))_{k\in \N}$ and $(\theta^k_{best})_{k\in \N}$, for $k$ sufficiently large, have the following convergence rates 
$$
  |\theta^{FW}(x^k)|\approx \mathcal{O}\left(\frac{1}{k^\frac{q}{2(q-2)}}\right),\quad \theta^k_{best}\approx \mathcal{O}\left(\frac{1}{k^\frac{q-1}{q-2}}\right), \quad \mbox{for } q>2.
$$
For the particular case in which $q=2$, these convergence are further improved. Indeed,  from  the first inequality in \eqref{ineq:basicRecursive-thetatilde}, \eqref{convrate-thetak-3} and the assumption $|\tilde \theta(x^k)|\geq c>0$, we have 
$$
 \min\left\{1,\frac{\alpha c}{4L}\right\}\sum_{i=m}^k|\theta^{FW}(x^i)|\leq 2\hat h(x^m), \quad \forall k\geq m\geq 1,
$$
which, combined with \eqref{rt65} with $q=2$ and the fact that $\hat h(x^1)\leq L\D_\X^2/2$ (see \eqref{boundh}), implies that  $$
\sum_{i=m}^k|\theta^{FW}(x^i)|\leq 2\max\left\{1,\frac{4L}{\alpha c}\right\}\frac{LD_{\X}^2}2\max\left\{\frac12,1-\frac{\zeta}{L}\right\}^{m-1}, \quad \forall k\geq m\geq 1.
$$
The above inequality  with $m=k$ and $m=\lfloor k/2 \rfloor$ implies, respectively, the improved convergence rates 
\begin{align}
|\theta^{FW}(x^k)|\approx \mathcal{O}\left(\beta^k\right),\quad 
\theta^k_{best} \approx \mathcal{O}\left(\frac{\beta^k}{k}\right),
\end{align}
where $\beta:=\max\left\{\frac12,1-\frac{\zeta}{L}\right\}<1$. The  above convergence  rates  for  $(\theta^k_{best})_{k\in \N}$ are better than  the rate of ${\cal O}(1/k)$  established in  {{\cite[Theorem~16]{Leandro2022}}}  under assumption {\bf (A1)}. It is worth mentioning that the latter reference does not contain convergence rate for $(\theta(x^k))_{k\in \N}$.
\end{remark}

Next we discuss some properties about the  quantity $(\tilde \theta(x^k))_{k\in \N}$ computed in \eqref{def:rk} and present some conditions under which the assumption considered on this sequence in Theorem~\ref{theorem:Xunifconv} holds.
\begin{remark}\label{rem:thetatildexk}
(i) Let $\|\cdot\|_2$ denote the norm induced by the inner product $\langle \cdot, \cdot\rangle$ in ${\mathcal Y}$.  It is easy to see that in the scalar setting, i.e., $m=1$, we have $\tilde \theta(x^k)=-\|\nabla f(x^k)\|_2$. Hence, in this case, the assumption on the  quantity $\tilde \theta(x^k)$ considered in Theorem~\ref{theorem:Xunifconv} coincides with the one in \cite{kerdreux2020,survey2022,garber2015faster}. (ii) For a given point $x\in \mathbb{R}^n$, the multiobjective steepest descent direction at $x$, see \cite{benar&fliege}, is defined as
$$
d^s(x):=\argmin_{d\in \mathbb{R}^n} \max_{j \in \J} \langle \nabla F_j(x),d\rangle +\frac{1}{2}\|d\|_2^2.
$$
It can be shown, see \cite[Corollary 2.3]{SVAITER2018430}, that  $\max_{j \in \J}\inner{\nabla F_j(x)}{d^s(x)}= -\|d^s(x)\|_2^2.$ Now, since ${\mathcal Y}$ is a finite dimensional space, there exists $\kappa>0$ such that $\|
x\|_2\geq \kappa \|x\|$ for all $x\in {\mathcal Y}$. Hence, for every $k\geq 0$,
\[
\tilde \theta(x^k)=\min_{\|z\|\leq 1}\max_{{j \in \J  }}\inner{\nabla F_j(x^k)}{z}
\leq \max_{j \in \J  } \left\langle {\nabla F_j(x^k)},{\frac{d^s(x^k)}{\|d^s(x^k)\|}}\right\rangle=-\frac{\|d^s(x^k)\|^2_2}{\|d^s(x^k)\|}\leq -\kappa\|d^s(x^k)\|_2,
\]
which implies that 
\begin{equation}\label{56lo}
|\tilde \theta(x^k)|\geq \kappa\|d^s(x^k)\|_2.
\end{equation}
Thus, the assumption on $(\tilde \theta(x^k))_{k\in \N}$ in Theorem~\ref{theorem:Xunifconv} holds if it is assumed that $\inf_{k\in  \N}\|d^s(x^k)\|_2>0$. 
Additionally, it is worth  pointing out that the latter assumption can be ensured by requiring that the constraint set $\X$ does not  contain any  unconstrained weak Pareto point of the objective function $F$. Indeed, this observation follows from the fact that $(x^k)_{k \in \N}\subset \X$, $\X$ is compact, $d^s(\cdot)$ is a continuous function and $d^s(x)=0$ if and only if $x$ is an unconstrained weak Pareto point of $F$, see \cite[Lemma~1]{benar&fliege}.
\end{remark}

\section{Illustrative examples}\label{sec:experiments}

In this section, we present some instances of \eqref{vectorproblem} to illustrate the behavior of the M-FW algorithm under the scenarios described in Theorems~\ref{theo:interior}, \ref{uniformly convex2} and \ref{theorem:Xunifconv}. These examples provide concrete situations to show how different set of assumptions considered in  the paper can affect the performance of the method.


We begin by illustrating that assuming only the strong convexity of  $F$ is insufficient to improve the sublinear convergence rate of the M-FW algorithm.

Let 
\begin{equation} \label{eq:o0}
 F(x) = \frac{1}{2}( \|x - b \|_2^2 , \quad \| x - c \|_2^2)^T,   
\end{equation}
and the feasible constraint set 
\[
\X = \{ x \in \mathbb{R}^2 \, : \,  \| x \|_1 \leq 1 \}.
\]
It is easy to see that, regardless the points $b, c \in \R^2$, we have $\mu = L = 1$ and $D_{\X} = 2$.

Consider $b = (-0.6, -0.6)^T$ and $c = (-0.5, -0.5)^T$. It is not hard to see that the unconstrained Pareto set intersects the feasible set $\X$ in a single boundary point $(-0.5, -0.5)^T$ (see Figure~\ref{fig:example1a}(left)). In this case, we observe the sublinear convergence rate of the sequence $( \hat{h}(x^k))_{k\in \N}$ generated by the M-FW algorithm (Figure~\ref{fig:example1a}(right)). 

Now, consider the same $b$ as before, but $c = (-0.01,-0.01)^T$ instead. Since $(-0.01,-0.01)^T$ is a Pareto point in the relative interior of $\X$ and the objectives are strongly convex, Theorem~\ref{theo:interior} applies and the linear convergence rate of $( \hat{h}(x^k))_{k\in \N}$ is illustrated in Figure~\ref{fig:example1b}(right). 

\begin{figure}
    \centering
    \includegraphics[scale=0.34]{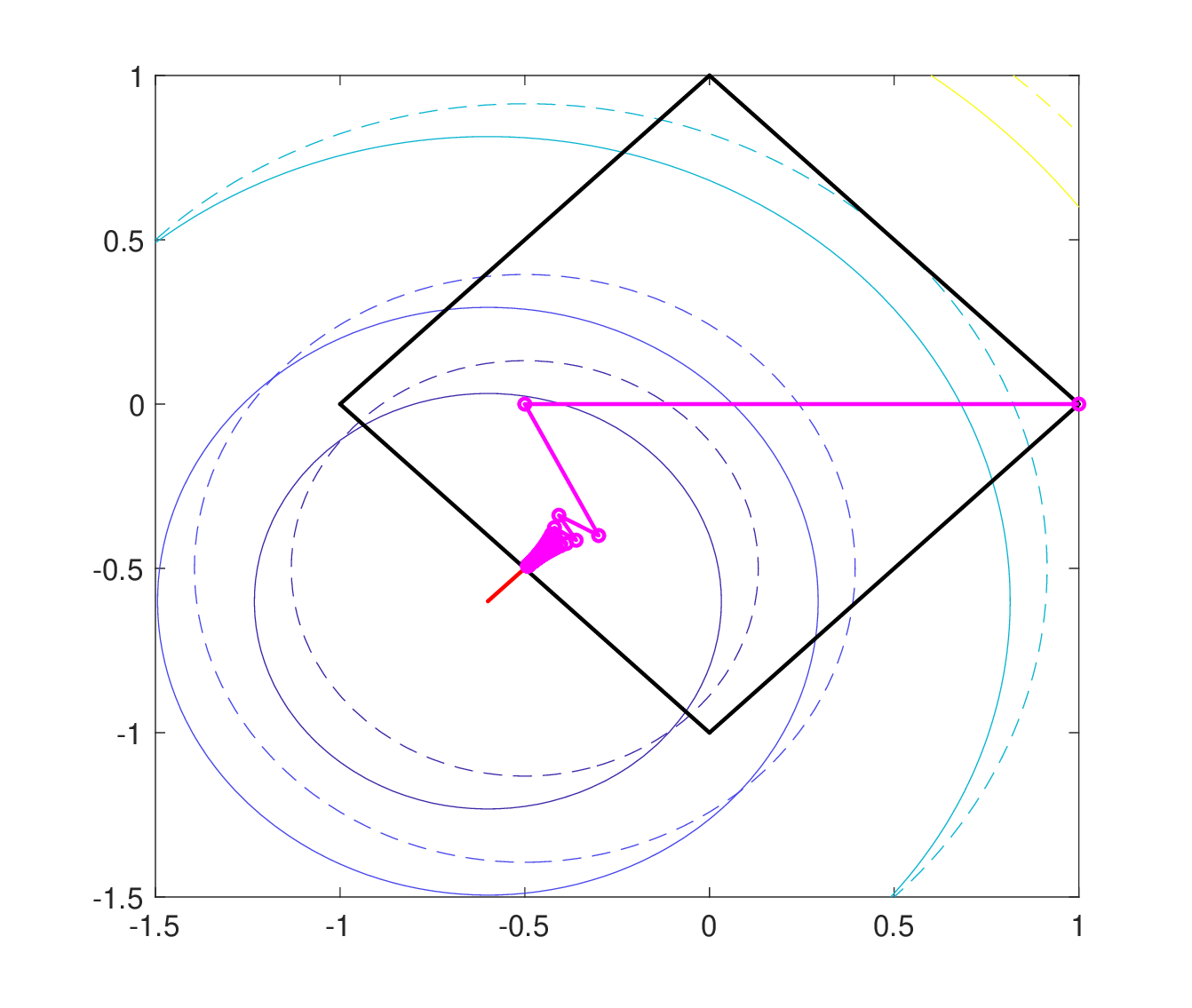}
    \includegraphics[scale=0.34]{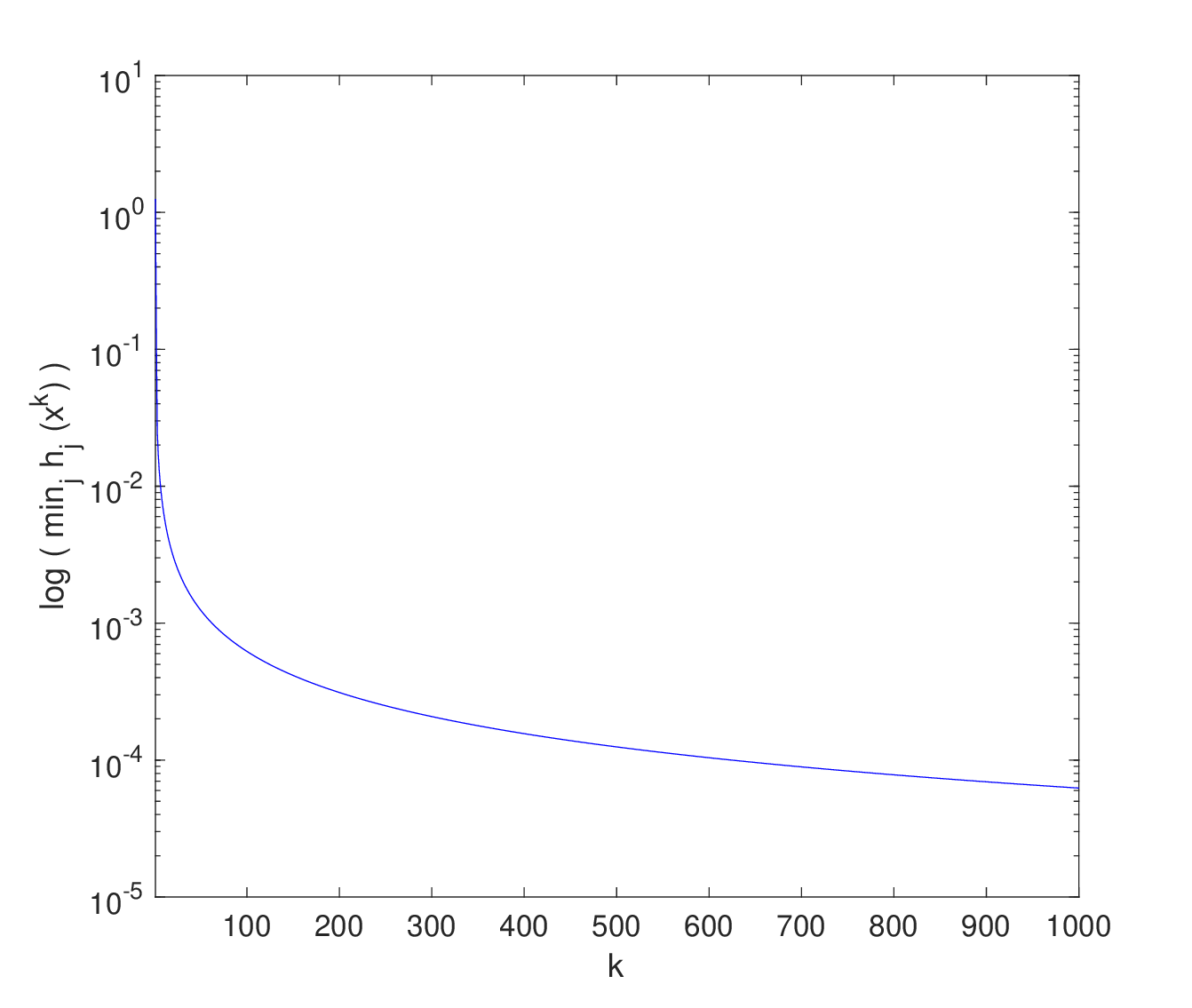}
    \caption{On the left we illustrate the unconstrained Pareto set (red line segment), the level curves of the two objectives, the constraint feasible set, and the trajectory of M-FW. On the right, the sublinear convergence rate is illustrated by plotting $\log \hat{h}(x^k)$ versus the number of iterations.}
    \label{fig:example1a}
\end{figure}

\begin{figure}
    \centering
    \includegraphics[scale=0.34]{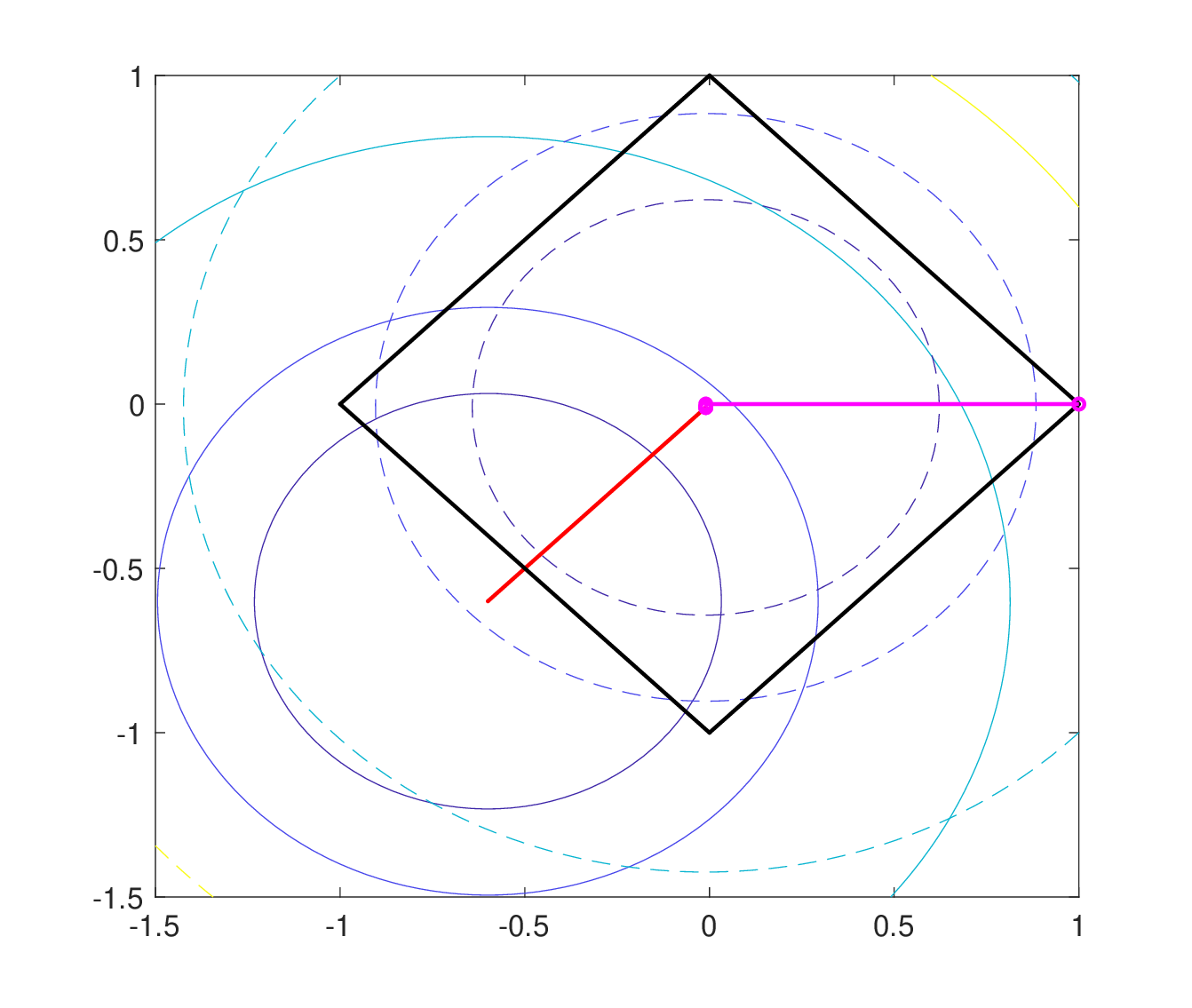}
    \includegraphics[scale=0.34]{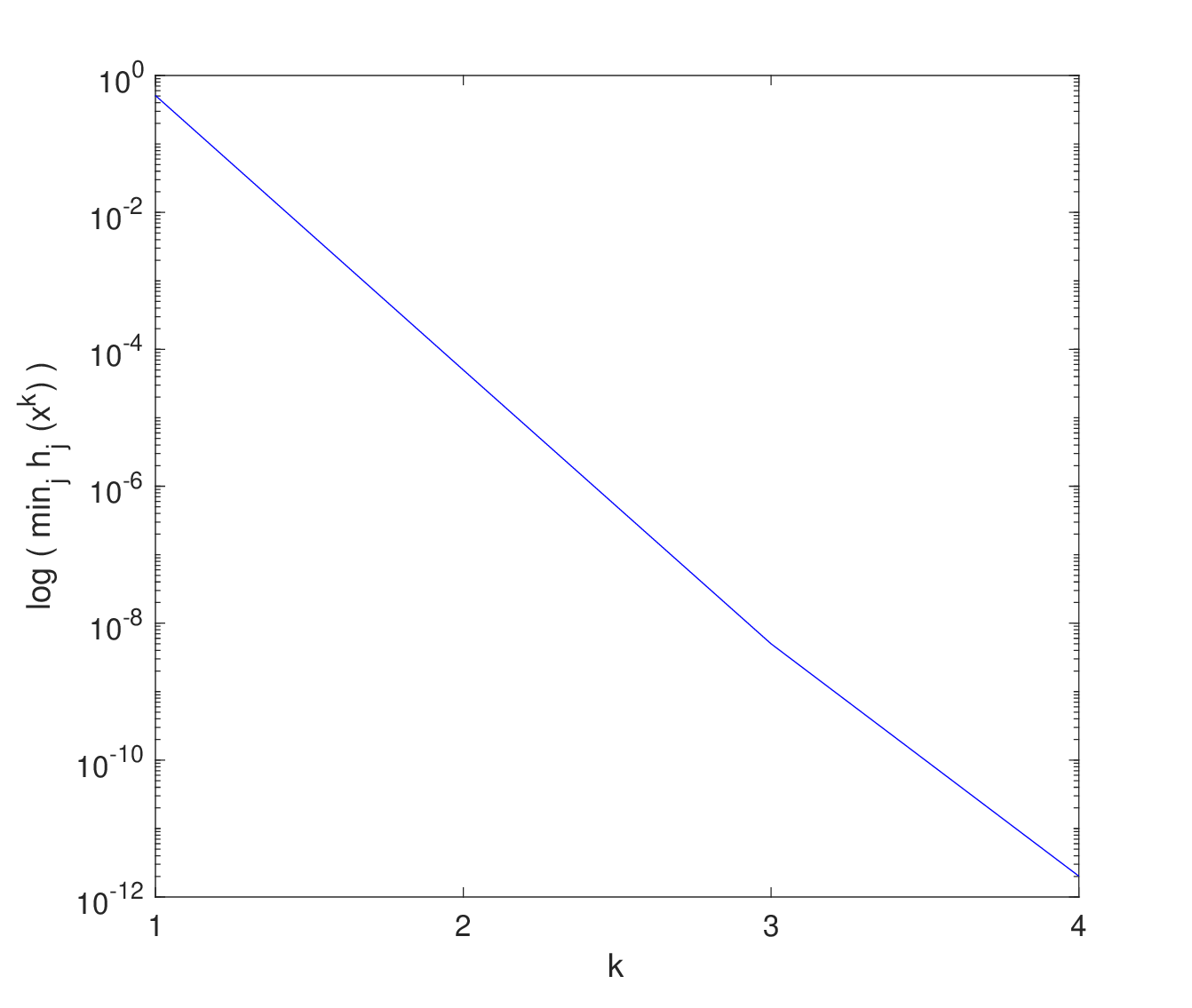}
    \caption{On the left we illustrate the unconstrained Pareto set (red line segment), the level curves of the two objectives, the constraint feasible set, and the trajectory of M-FW. On the right, the linear convergence rate is illustrated by plotting $\log \hat{h}(x^k)$ versus the number of iterations.}
    \label{fig:example1b}
\end{figure}

In order to illustrate Theorem~\ref{uniformly convex2}, 
 consider $F$ as in \eqref{eq:o0}, but  the feasible region as
\[
\X = \{ x \in \mathbb{R}^n \, : \,  \| x \|_2 \leq 1 \}, 
\]
which is an ($1$,$2$)-uniformly convex set.  
Take  $b=(-1/\sqrt{2}, -1/\sqrt{2})^T$ and $c=(-3/4, -3/4)^T$. In this setting, since {\bf (A2)} holds and $q=2$, we should observe an ${\cal O}(1/k^2)$ convergence rate as in Theorem~\ref{uniformly convex2}. This is illustrated in Figure~\ref{fig:example3}. 

\begin{figure}
    \centering
    \includegraphics[scale=0.3]{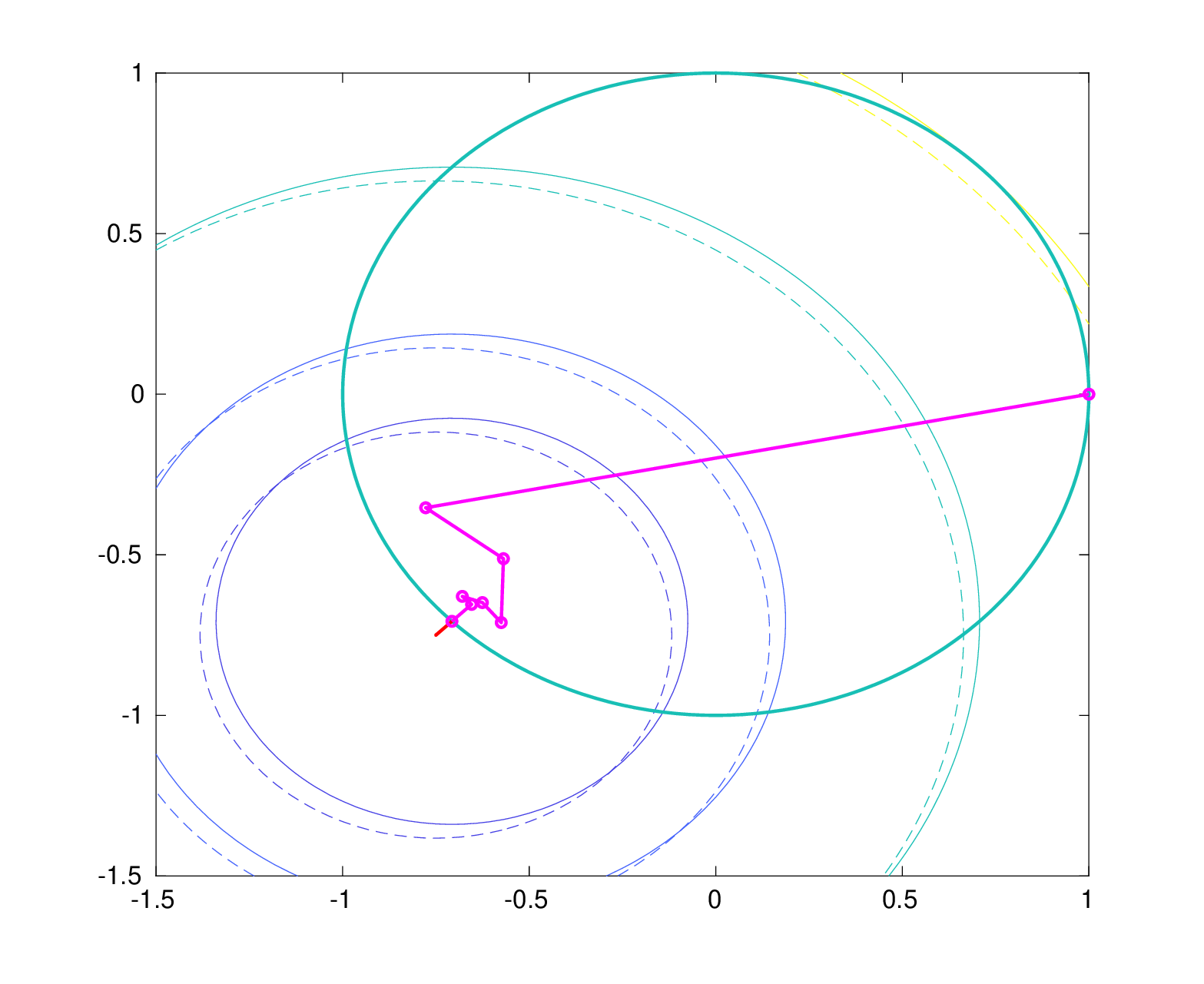}
    \includegraphics[scale=0.3]{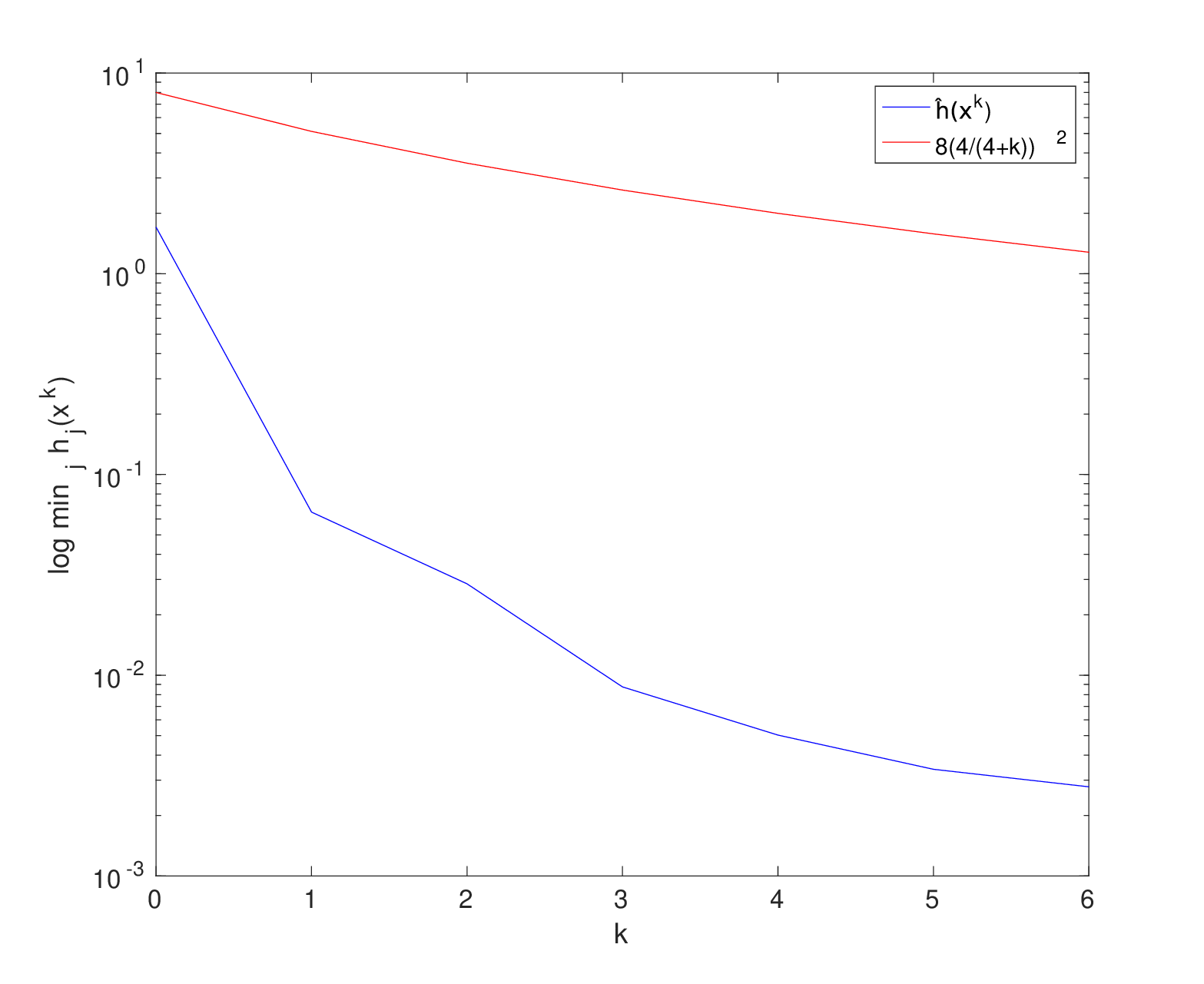}
    \caption{On the left we illustrate the unconstrained Pareto set (red line segment), the level curves of the two objectives, the constraint feasible set, and the trajectory of M-FW. On the right, the ${\cal O}(1/k^2)$ rate is illustrated by plotting $\log \hat{h}(x^k)$ versus the number of iterations.}
    \label{fig:example3}
\end{figure}

Finally, recall that  Theorem~\ref{theorem:Xunifconv} does not require the strong convexity of the objectives, only the ($\alpha$,$q$)-uniform  convexity of the feasible set. However, the extra condition is that $\tilde{\theta}(x)$ stays bounded away from zero, at least at the iterates generated by M-FW. 
Consider now
\[
F(x) = \frac{1}{2}( \|Ax - b_1 \|_2^2 , \quad \| Ax - b_2 \|_2^2)^T,
\]
where $A = e_1 e_1^T$, $e_1 = (1, 0)^T$, $b_1 = (-1.1,0)^T$ and $b_2 =(-1.3,0)^T$, with 
\[
\X = \{ x \in \mathbb{R}^2 \, : \,  \| x \|_2 \leq 1 \}
\]
as feasible region. Since $A$ is only positive semidefinite, the objectives are convex but not strongly convex. As in the previous example, $\X$ is (1,2)-uniformly convex.  

Thus, according to Theorem~\ref{theorem:Xunifconv} we should expect a linear convergence rate (since $q=2$) as long as  $c:=\inf_k|\tilde \theta(x^k)|>0$. In order to estimate $c$, we observe that $\langle \nabla F_j(x), z \rangle = (x_1 - b_{j1})z_1$, for $j=1,2$. Also, $x_1 - b_{j1} \geq -1 - b_{11} = 0.1 > 0$, for $x \in \X$. Hence, the minimizer of $\max\{\langle \nabla F_1(x), z \rangle, \langle \nabla F_2(x), z \rangle\}$ for $z \in \X$ is $z=(-1,0)^T$ with optimal value $\max \{ b_{11} - x_1, b_{21} - x_1 \} = b_{11} - x_1 = -1.1 - x_1 \geq 0.1$ for $x \in \X$, so we obtain that $c \geq 0.1 > 0$.
Figure~\ref{fig:example4} illustrates the linear convergence rate of the M-FW algorithm in this example. 

\begin{figure}
    \centering
    \includegraphics[scale=0.3]{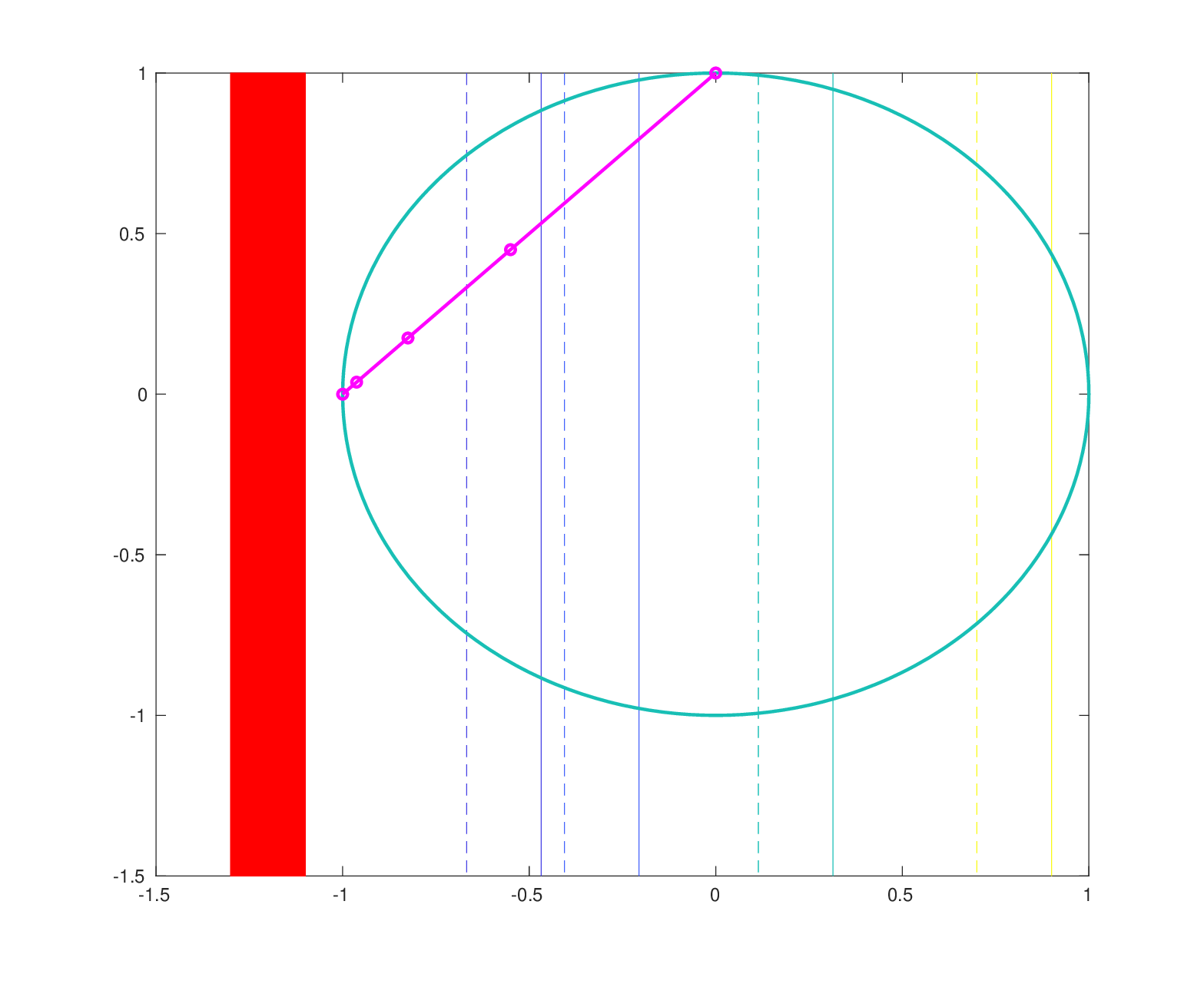}
    \includegraphics[scale=0.3]{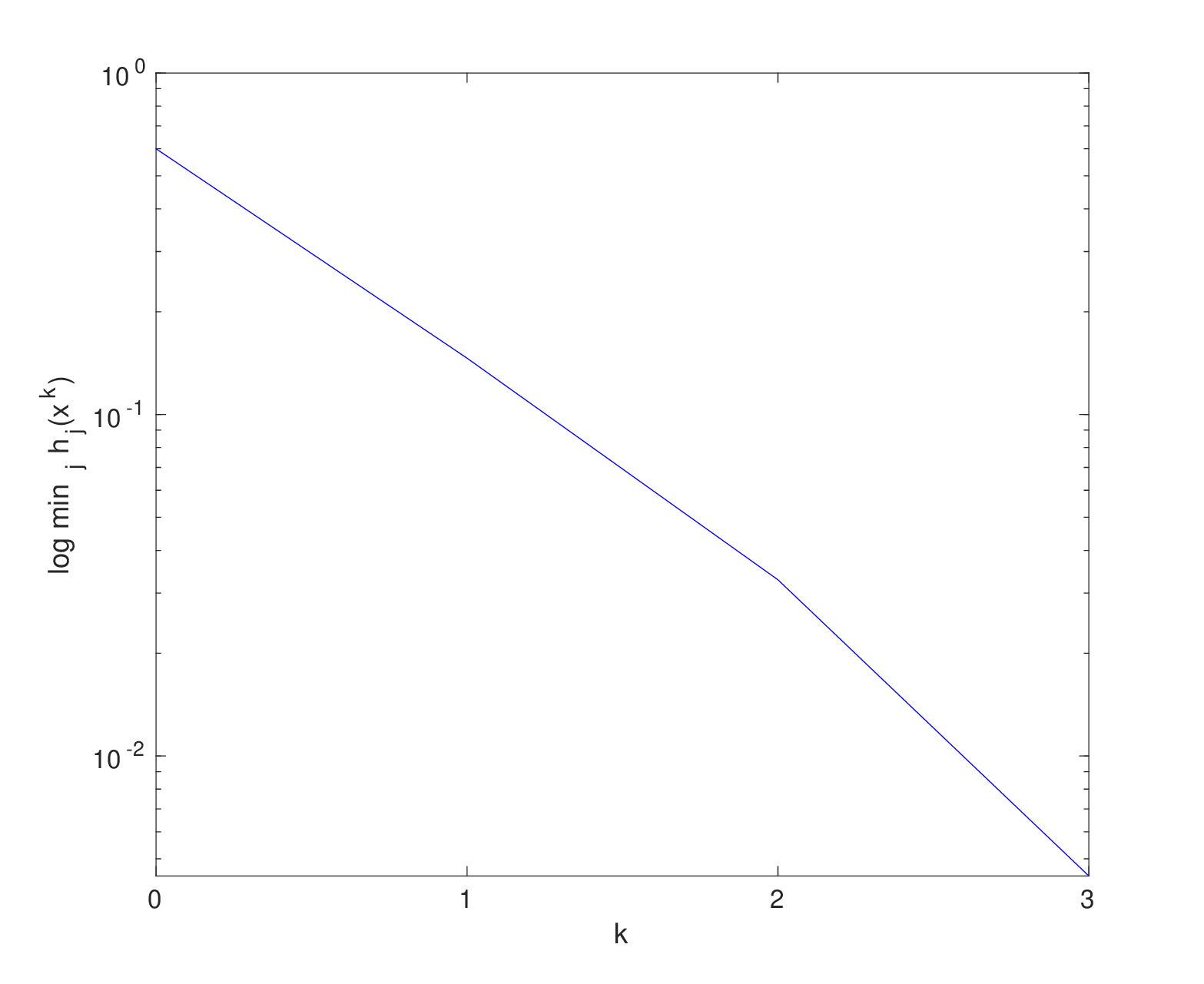}
    \caption{On the left we illustrate the unconstrained Pareto (red rectangle), the level curves of the two objectives, the constraint feasible set, and the trajectory of M-FW. On the right, the linear convergence rate is illustrated by plotting $\log \hat{h}(x^k)$ versus the number of iterations.}
    \label{fig:example4}
\end{figure}

\section{Conclusion}\label{sec:final}
This paper established convergence rates for the multiobjective Frank-Wolfe algorithm for solving constrained convex multiobjective optimization problems. It was shown that the algorithm achieves faster sublinear as well as linear convergence rates under different set of assumptions such as strong convexity of the multiobjective  function and uniform convexity of the constraint set. Illustrative  examples were presented to showcase the convergence rates obtained and the assumptions considered.

\section*{Data availability statement}
Data sharing not applicable – no new data generated, or the article describes entirely theoretical research.

\section*{Conflict of interest}
The authors have no Conflict of interest to declare that are relevant to the content of this article.

\bibliographystyle{abbrv}
\bibliography{Referencias}

\end{document}